\documentclass[oneside]{amsart}

\usepackage[T1]{fontenc}
\usepackage{geometry}			%%% customization of the page margins
\usepackage[parfill]{parskip}	%%% separate paragraphs by vertical space
\usepackage{csquotes}			%%% smart quotes

\usepackage{mathtools}		%%% extension of amsmath, which it loads internally
	\allowdisplaybreaks			%%% allow equation groups to break across the page
\usepackage{xparse}			%%% provides high-level interface for producing document-level commands
\usepackage{etoolbox}		%%% provides toolbox of programming tools, LaTeX frontends, etc.
\usepackage{xcolor}			%%% driver independent color extensions for \LaTeX\ and pdf\LaTeX

\usepackage{amsthm}					%%% theorem and proof environments (should be loaded after amsmath/mathtools)
\usepackage[shortlabels]{enumitem}	%%% allows customising lists (supersedes the `enumerate` package)
\usepackage{array}					%%% flexible column formatting in tables

\usepackage{amssymb}	%%% loads useful math symbols and fonts, including a blackboard font (loads amsfonts internally)

\usepackage{microtype}
\usepackage{tikz}				%%% creating diagrams
\usepackage{amsthm}			%%% theorem environments

\usepackage{varioref}		%%% for variable cross-reference names depending on where they appear
\usepackage[backref=page]{hyperref}	%%% for handling cross-referencing, mainly linking capabilities
\usepackage[capitalize,noabbrev]{cleveref}	%%% for intelligent cross-references
	\hypersetup{
		colorlinks=true,
		linkcolor=blue,
		filecolor=blue,
		urlcolor=blue,
	}
\providecommand{\given}{}

\newcommand{\SetSymbol}[1][]{%
      #1 :
      \allowbreak
      \mathopen{}
}
\newcommand{\defining}{\emph}			%%% markup for defining terms occurring in the text
\newcommand{\defn}{\coloneqq}			%%% markup for definition of an object (1)
\DeclareMathOperator{\chr}{\chi}			%%% \chr(G)		chromatic number of the graph G
\DeclarePairedDelimiter{\parens}{\lparen}{\rparen}		%%% parentheses		( )
\DeclarePairedDelimiter{\floor}{\lfloor}{\rfloor}		%%% floor function
\DeclarePairedDelimiter{\ceil}{\lceil}{\rceil}			%%% ceiling function
\DeclarePairedDelimiter{\card}{\lvert}{\rvert}			%%% set cardinality
\DeclarePairedDelimiterX\Set[1]\{\}%
	{%
		\renewcommand\given{\SetSymbol[\delimsize]}
		#1
	}

\theoremstyle{plain}
	\newtheorem*{theorem*}{Theorem}
	\newtheorem{theorem}{Theorem}[section]

\usepackage[abbrev]{amsrefs}			%%% AMS package for handling bibliography entries
\renewcommand{\MR}[1]{ \href{https://mathscinet.ams.org/mathscinet-getitem?mr=MR#1}{MR#1}}	%%% markup for MathSciNet MR reviews
\newcommand{\Zbl}[1]{ \href{https://zbmath.org/?q=an:#1}{Zbl~#1}}	%%% markup for zbMATH Zbl reviews
\newcommand{\JFM}[1]{ \href{https://zbmath.org/?q=an:#1}{JFM~#1}}	%%% markup for zbMATH JFM reviews
\title{Note on $4$-coloring $6$-regular triangulations on the torus}
\author{Brahadeesh Sankarnarayanan}
\address{Department of Mathematics, Indian Institute Technology Bombay, Powai, Mumbai 400076, Maharashtra, India.}
\email{bs@math.iitb.ac.in}
\date{June 3, 2021} 
\thanks{Research supported by the National Board for Higher Mathematics (NBHM),
Department of Atomic Energy (DAE), Govt.\ of India.\\
This preprint has not undergone peer review or any post-submission improvements or corrections. The Version of Record of this article is published in \emph{Ann Comb} \textbf{26}(3), 559--569 (2022), and is available online at \url{https://doi.org/10.1007/s00026-022-00573-8}.}

 %2020

\begin{document}

\begin{abstract}
	In 1973, Altshuler~\cite{Alt73} characterized the $6$-regular triangulations on the torus to be
	precisely those that are obtained from a regular triangulation of the $r \times s$
	toroidal grid where the vertices in the first and last column are
	connected by a shift of $t$ vertices. Such a graph is denoted $T(r, s, t)$.
	
	In 1999, Collins and Hutchinson~\cite{ColHut99} classified the $4$-colorable
	graphs $T(r, s, t)$ with $r, s \geq 3$. In this paper, we point out a gap in their 
	classification and show how it can be fixed.
	Combined with the classification of the $4$-colorable graphs $T(1, s, t)$
	by Yeh and Zhu~\cite{YehZhu03} in 2003, this completes
	the characterization of the colorability of all the $6$-regular triangulations on the torus.
\end{abstract}

\maketitle

\section{Introduction}\label{S:Introduction}

A classical result due to Heawood~\cite{Hea90} states that the chromatic number
of any graph embeddable on an orientable surface of genus \(g > 0\) is bounded
above by the so-called \defining{Heawood number} \(H(g) \defn \floor{(7 + \sqrt{1 + 48g})/2}\).
Heawood also showed that this upper bound is tight for \(g = 1\) by exhibiting an embedding of \(K_{7}\)
in the torus. In fact, \(K_{7}\) embeds in the torus as a triangulation, that is,
an embedding in which each face is homeomorphic to a disc and is bounded by three edges.
It follows from Euler's formula that every regular triangulation
on the torus (of which \(K_{7}\) is one example) is necessarily \(6\)-regular,
and that every \(6\)-regular graph embeddable on the torus is necessarily
a triangulation.

Altshuler~\cite{Alt73} classified the \(6\)-regular triangulations of the torus
as follows (we follow the notation in~\cite{BalSan21}).
For integers \(r \geq 1\), \(s \geq 1\) and \(0 \leq t \leq s - 1\),
take \(V = \Set{ (i, j) \given 1 \leq i \leq r, 1 \leq j \leq s }\)
to be the vertex set of the graph \(T(r, s, t)\) equipped with the following edges:
\begin{itemize}
	\item For each \(1 < i < r\), \((i, j)\) is adjacent to
	\((i, j \pm 1)\), \((i \pm 1, j)\) and
	\((i \pm 1, j \mp 1)\).
	
	\item If \(r > 1\), \((1, j)\) is adjacent to \((1, j \pm 1)\), \((2, j)\),
	\((2, j - 1)\), \((r, j + t + 1)\) and \((r, j + t)\). 
	
	\item If \(r > 1\), \((r, j)\) is adjacent to \((r, j \pm 1)\),
	\((r - 1, j + 1)\), \((r - 1, j)\), \((1, j - t)\) and \((1, j - t - 1)\).
	
	\item If \(r = 1\), \((1, j)\) is adjacent to \((1, j \pm 1)\), \((1, j \pm t)\)
	and \((1, j \pm (t + 1))\).
\end{itemize}
Here, addition in the first coordinate is taken modulo \(r\) and
in the second coordinate is taken modulo \(s\).
\Cref{F:T(562)} depicts the graph \(G = T(5, 6, 2)\);
note that the edges between the top and bottom rows are not shown.

\begin{figure}
\centering
	\begin{tikzpicture}[font=\scriptsize]
	\draw (1,1) grid (6,6);
	\foreach \y in {2,...,6}
		\foreach \x in {1,...,5}{
			\draw (\x,\y) -- (\x+1,\y-1);
	}
	\foreach \x in {1,...,6}
		\foreach \y in {1,...,6}{
			\draw[fill=black] (\x,\y) circle(2pt);
		}
	\foreach \y in {1,...,6}{
		\node[anchor=east] at (1,\y) {(1,\y)};
		\pgfmathparse{Mod(\y-5,6)}
		\node[anchor=west] at (6,\pgfmathresult+1) {(1,\y)};
	}
	\foreach \x in {2,...,5}{
		\node[anchor=north] at (\x,1) {(\x,1)};
		\node[anchor=south] at (\x,6) {(\x,6)};
	}
	\end{tikzpicture}
	\caption{\(G = T(5,6,2)\)}\label{F:T(562)}
\end{figure}
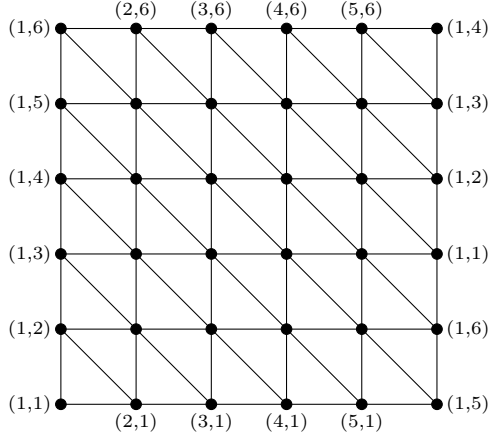

It is clear that each \(T(r, s, t)\) is a \(6\)-regular triangulation
of the torus. Altshuler's theorem says that these
are all the \(6\)-regular triangulations on the torus up to isomorphism (similar
constructions also appear in \cites{Neg83,Tho91}).

\begin{theorem}[Altshuler~\cite{Alt73}, 1973]\label{T:Altshuler}
	Every \(6\)-regular triangulation on the torus is isomorphic
	to \(T(r, s, t)\) for some integers \(r \geq 1\), \(s \geq 1\), and \(0 \leq t < s\).
\end{theorem}

As shown by Altshuler in~\cites{Alt72,Alt73},
through every vertex \(v\) of \(T(r, s, t)\) there are three so-called
\defining{normal circuits}, which are the simple cycles obtained
by traversing through \(v\) along each of the three directions---vertical, horizontal,
and diagonal---in the natural fashion. These normal circuits
have lengths \(s\), \(\frac{n}{\gcd(s, t)}\), and \(\frac{n}{\gcd(s, r + t)}\), respectively,
where \(n = rs\) is the order of \(T(r, s, t)\).

By picking a different normal circuit to be represented
as the vertical cycle, one can see that there exist
\(0 \leq t_{1} < \frac{n}{\gcd(s, t)}\) and \(0 \leq t_{2} < \frac{n}{\gcd(s, r + t)}\)
such that \(T(r, s, t)\) is isomorphic to \(T\parens[\big]{\gcd(s, t), \frac{n}{\gcd(s, t)}, t_{1}}\)
as well as to \(T\parens[\big]{\gcd(s, r + t), \frac{n}{\gcd(s, r + t)}, t_{2}}\).
By swapping the horizontal and diagonal normal circuits,
one can see that \(T(r, s, t)\) is isomorphic
to \(T(r, s, t')\) for \(0 \leq t' < s\) such that \(t' \equiv - r - t \pmod{s}\).

Now, Dirac's map color theorem~\cite{Dir52c} states that the only
connected graph \(G\) with \(\chr(G) = H(g)\) that is embeddable on a surface
of genus \(g > 0\) is \(K_{H(g)}\). So, \(K_{7}\) is the only \(7\)-chromatic
\(6\)-regular triangulation on the torus.
Albertson and Hutchinson~\cite{AlbHut80} showed that there is a unique \(6\)-chromatic
\(6\)-regular simple triangulation on the torus, which has \(11\) vertices.
Then, Collins and Hutchinson~\cite{ColHut99} gave a characterization of the \(4\)-colorable
triangulations \(T(r, s, t)\) with \(r, s \geq 3\) as follows:

\begin{theorem}[Collins--Hutchinson~\cite{ColHut99}*{Theorem 1.2}]\label{T:Col-Hut}
Let \(G = T(r, s, t)\). If \(r, s \geq 3\), then
\(G\) can be \(4\)-colored, with a finite number of exceptions.
\end{theorem}

Note that a \defining{\(6\)-regular right-diagonal (unshifted) \(m \times n\) grid}
in their notation is what we call \(T(n, m, 0)\), and that a \defining{\(6\)-regular right-diagonal
\((m \times n; k)\) grid} for \(k > 1\) in their notation is what we call \(T(n, m, m - k + 1)\).
In particular, the \((m \times n; 1)\) grid is the same as the unshifted \(m \times n\) grid,
which is \(T(n, m, 0)\) in our notation.

In this paper, we point out a gap in the proof of \cref{T:Col-Hut}
that makes the statement incorrect, and we
provide a patch to the statement and proof.
In \cref{S:prelim}, we locate the error in the proof of \cref{T:Col-Hut},
and provide explicit counterexamples to its statement.
In \cref{S:main}, we prove the following modification of
the above theorem:
\begin{theorem}\label{T:main}
Let \(G = T(r, s, t)\) be a simple \(6\)-regular triangulation having normal circuits of lengths
\(a \geq b \geq c\). Suppose that \((\frac{n}{a}, \frac{n}{b}) \neq (1, 1), (1, 2)\),
where \(n = rs\) is the order of \(G\). Then \(G\) can be \(4\)-colored.
\end{theorem}

Combined with the earlier results~\cites{Hea90,Dir52c,AlbHut80,ColHut99}
as well as the classification of the \(4\)-colorable
triangulations \(T(1, s, t)\) by Yeh--Zhu~\cite{YehZhu03}, we
complete the characterization of the colorability of all the
\(6\)-regular triangulations on the torus in \cref{T:final} in \cref{S:conclusion}.

\section{Examining the statement of \texorpdfstring{\cref{T:Col-Hut}}{Theorem~1.2} and its proof}\label{S:prelim}

\subsection{Constructing counterexamples to \texorpdfstring{\cref{T:Col-Hut}}{Theorem~1.2}}\label{SS:counter}

Collins and Hutchinson identify that \(T(3, 3, 1)\), \(T(3, 3, 2)\),
\(T(3, 5, 3)\), and \(T(3, 5, 4)\) are not \(4\)-colorable, but state
that there are no others of the form \(T(3, s, t)\) for \(s \geq 3\)~\cite{ColHut99}*{Theorem~3.7}.
However, as we show below, the graphs of the form \(T(3, s, s - 2)\) and \(T(3, s, s - 1)\)
are not \(4\)-colorable for all \(s \not\equiv 0 \pmod{4}\). Note
that the four graphs mentioned in the beginning are obtained by plugging
in \(s = 3, 5\) in these expressions.

Now, consider the triangulations \(T(1, s, 2)\) for \(s \geq 7\). These are
simple graphs, and, as noted in \cite{ColHut99}*{Section~3}, every four successive vertices of \(T(1, s, 2)\)
induce a \(K_{4}\). Thus, \(T(1, s, 2)\) is \(4\)-colorable for \(s \geq 7\) if and only if
\(s \equiv 0 \pmod{4}\).

Therefore, we consider the graphs \(T(1, 3s, 2)\) for \(s \geq 3\) and \(s \not\equiv 0 \pmod{4}\).
These are all \(5\)-chromatic graphs by \cites{Hea90,Dir52c,AlbHut80}.
The normal circuits in \(T(1, 3s, 2)\) have lengths \(3s\), \(3s\), and \(s\), respectively,
so \(T(1, 3s, 2)\) is isomorphic to \(T(3, s, t)\) for some \(0 \leq t \leq s - 1\).
Since there are infinitely many \(s \geq 3\) such that \(s \not\equiv 0 \pmod{4}\),
there are infinitely many graphs of the form \(T(3, s, t)\)
that are not \(4\)-colorable, contradicting the statement of \cref{T:Col-Hut}.

In fact, one can check by a careful computation that \(t = s - 2, s - 1\) in this case.
For simplicity, we label the vertex \((1, j)\) with the integer \(j\)
(recall that \(j\) is taken modulo \(3s\)).
Map the vertical, horizontal, and diagonal normal circuits of \(T(1, 3s, 2)\)
to the horizontal, diagonal, and vertical normal circuits of \(T(3, s, 2)\),
respectively.
Then, the vertical normal circuit of \(T(3, s, t)\) has labels \(3s, 3s-3, 3s-6, 3s-9, \dotsc\)
from top to bottom when drawn as in~\cref{F:T(562)}. The horizontal normal circuit
through the vertex with label \(3s\) has the first four labels in the right-to-left
direction as \(3s\), \(3s - 1\), \(3s - 2\), and \(3s - 3\).
Thus, the shift is \(t = s - 1\). Since \(T(r, s, t)\) is isomorphic to \(T(r, s, t')\)
for \(0 \leq t' < s\) such that \(t' \equiv -r - t \pmod{s}\), the graph
\(T(3, s, s - 1)\) is isomorphic to \(T(3, s, s - 2)\). Thus,
the graphs \(T(3, s, s - 1)\) and \(T(3, s, s - 2)\) are not \(4\)-colorable
for all \(s \geq 3\) such that \(s \not\equiv 0 \pmod{4}\).

\subsection{Gap in the proof of \texorpdfstring{\cref{T:Col-Hut}}{Theorem~1.2}}\label{SS:gap}

The proof of \cref{T:Col-Hut} in \cite{ColHut99} is broken up into a sequence of results,
first for the unshifted triangulations \(T(r, s, 0)\) \cite{ColHut99}*{Lemma~3.2, Theorem~3.3, Lemma~3.4},
and then for the shifted triangulations \(T(r, s, t)\) with \(t \neq 0\) \cite{ColHut99}*{Theorems~3.6 and 3.7}.
We identify the following theorem as the source of the contradiction:

\begin{theorem}[Collins--Hutchinson~\cite{ColHut99}*{Theorem~3.6}]\label{T:problem}
	Let \(G = T(r, s, t)\) for some \(0 < t \leq s - 1\). Then if
	\(3 \leq s, r\), \(G\) can be \(4\)-colored except possibly in the case when
	\(r = 5\), or when \(t = s - 1\) and \(r = s\) or \(s + 1\), or when \(t = s - 2\)
	and \(r = s\).
\end{theorem}

The proof of this theorem proceeds as follows. Let \(C_{i}\) denote
the \(i\)th column of \(T(r, s, t)\), for \(i = 1,\dotsc,r\).
First, a proper \(4\)-coloring of \(T(y, s, 0)\) is used to color \(C_{1},\dotsc,C_{y}\),
where \(y \geq 3\) is to be determined.
(Note that \(y \geq 3\) is needed to ensure that a proper \(4\)-coloring
of \(T(y, s, 0)\) can be found using \cite{ColHut99}*{Theorem~3.3}.)
Then, the coloring on \(C_{1}\) is repeated on \(C_{y + 1}\). Then, the coloring on \(C_{y + 1}\) is
repeated on \(C_{y + 2}\) after an upward shift by one vertex.
As shown in \cite{ColHut99}*{Lemma~3.1}, this coloring on \(C_{y + 2}\)
is compatible with the coloring on \(C_{y + 1}\) as long as \(r, s \neq 5\).
Repeat this process of repeating the coloring on successive columns with an upward shift
to color the \(t\) columns \(C_{y + 2}, \dotsc, C_{y + t + 1}\).
Now, note that the coloring on \(C_{y + t + 1}\) is identical to the coloring
on \(C_{1}\), except that it is shifted upwards by \(t\) vertices.
Thus, this gives a valid coloring of \(T(r, s, t)\) provided that \(r = y + t\).

At this point, Collins and Hutchinson state that the inequality \(y \geq 3\)
fails only when \(t = s - 1\) and \(r = s\) or \(s + 1\), or when \(t = s - 2\)
and \(r = s\), so this concludes their proof. However, this conclusion can only be drawn under the additional
hypothesis that \(r \geq s\). Thus, their proof holds only under the additional
hypothesis that \(r \geq s\).

Hence, the statements
of \cref{T:problem} and \cite{ColHut99}*{Theorem~3.7} need to be modified
by adding the hypothesis that \(r \geq s\). However,
the statement of \cite{ColHut99}*{Theorem~3.8} is now weakened, since the colorability
of the triangulations \(T(2, s, t)\) (for odd \(s\)) that are not isomorphic to \(T(1, 2s, t')\)
for any \(0 \leq t' < 2s\) is no longer completely settled by their previous results.

Furthermore, the above method of proof does not seem to easily lend itself to the cases
when \(r < s\). The above argument does extend to the graphs \(T(r, s, t)\),
\(3 \leq r, s\), \(r \neq 5\), \(s \neq 5\), with \(r \geq t + 3\) or \(r > s - \ceil{t/2}\),
the latter by extending the coloring on \(C_{y}\) by
using downward shifts by two vertices instead of upward shifts by one vertex,
but it is not clear, for instance, how one can \(4\)-color \(T(10,990,100)\)
by an argument along the above lines.

In the next section, we provide
a different argument to color the shifted triangulations.

\section{Main result}\label{S:main}

We shall use Collins and Hutchinson's coloring of the unshifted triangulations~\cite{ColHut99}*{Theorem~3.3}
to color the shifted triangulations.

\begin{theorem}\label{T:part1}
	Let \(G = T(r, s, t)\) be a simple triangulation with
	\(r \neq 1\) and \(s\) as the maximal length of a normal
	circuit in \(G\). Then \(G\) is \(4\)-colorable.
\end{theorem}
\begin{proof}
	Note that the conditions on \(r\) and \(s\) imply that \(\gcd(s, t), \gcd(s, r + t) \geq r \geq 2\).
	
	Suppose that \(r \geq 3\). Now, use a proper coloring of \(T(r, \gcd(s, t), 0)\)
	to color the first \(\gcd(s, t)\)-many rows of \(T(r, s, t)\), and repeat
	this coloring block \((s/\gcd(s, t))\)-many times to get a coloring
	of \(T(r, s, t)\). For this coloring to be proper, the coloring on the column \(C_{r}\)
	is required to be compatible with an upward shift by \(t\) vertices of the
	coloring on the column \(C_{1}\). But, since the coloring on the column \(C_{1}\)
	is periodic with period \(\gcd(s, t)\), an upward shift by \(t\) vertices
	of the coloring on \(C_{1}\) is the same as no shift. Thus, we only need
	to check that the coloring on \(C_{r}\) is compatible with that on \(C_{1}\),
	and this holds since it is obtained from a periodic coloring of the unshifted
	triangulation \(T(r, \gcd(s, t), 0)\).
	
	Next, suppose that \(r = 2\). As observed in \cite{ColHut99}*{Theorem~3.8},
	if \(s\) is even, then \(G\) is \(4\)-colorable,
	simply by \(2\)-coloring the columns \(C_{1}\) and \(C_{2}\)
	with the colors \(\Set{1, 2}\) and \(\Set{3, 4}\), respectively.
	
	So, suppose that \(r = 2\) and that \(s\) is odd, which imply that \(\gcd(s, t), \gcd(s, t + 2) \geq 3\).
	Thus, \(G\) is isomorphic
	to \(T(r', s', t')\) for \(r' = \gcd(s, t)\), \(s' = 2s/\gcd(s, t)\),
	and some \(0 \leq t' < s'\) such that \(\gcd(s', t') = \gcd(s, t + 2) \geq 3\).
	This is possible by the remarks following the statement of \cref{T:Altshuler}
	in \cref{S:Introduction}.
	Thus, we can repeat the previous algorithm to color \(T(r', s', t')\)
	as follows. First, use a proper coloring of
	\(T(r', \gcd(s', t'), t')\) to color the first \(\gcd(s', t')\)-many rows
	of \(T(r', s', t')\). Then, repeat this coloring block \((s'/\gcd(s', t'))\)-many
	times to get a coloring of \(T(r', s', t')\). This is verified to be a proper coloring
	for the same reason as in the case \(r \geq 3\),
	so this completes the case \(r = 2\) as well.
\end{proof}

In fact, the proof in the case \(r = 2\) and \(s\) odd shows that the following theorem is also true.

\begin{theorem}\label{T:part2}
	Let \(G = (V, E)\) be a simple \(6\)-regular triangulation on the torus
	with normal circuits of lengths \(a \geq b \geq c\)
	such that \(\frac{n}{c} \geq \frac{n}{b} \geq 3\), where \(\card{V} = n\). Then, \(G\) is \(4\)-colorable.
\end{theorem}

Specifically, a coloring of \(G\) can be found by viewing \(G\) as
\(T(\frac{n}{c}, c, t)\), where \(0 \leq t < c\) is such that
\(\gcd(c, t) = \frac{n}{b}\), and then coloring \(G\) by repeating
\((c/\gcd(c, t))\)-many times a proper coloring of \(T(\frac{n}{c}, \frac{n}{b}, 0)\).
This can be done since it is assumed that \(\frac{n}{c} \geq \frac{n}{b} \geq 3\).

\subsection*{Proof of \texorpdfstring{\cref{T:main}}{Theorem~1.3}}

Suppose that \(G = T(r, s, t)\) is a simple \(6\)-regular triangulation on the torus
with normal circuits of lengths \(a \geq b \geq c\)
such that \((\frac{n}{a}, \frac{n}{b}) \neq (1, 1), (1, 2)\), where \(n = rs\).

If \(\frac{n}{a} = 1\), then \(3 \leq \frac{n}{b} \leq \frac{n}{c}\),
so \(G\) is \(4\)-colorable by \cref{T:part2}.

If \(\frac{n}{a} = 2\), then \(G\) is isomorphic to \(T(2, a, t)\) for some \(0 \leq t < a\).
If \(\frac{n}{b} = 2\), then \(a\) is even, since \(b = n/\gcd(a, t)\) or \(n/\gcd(a, t + 2)\).
So \(G\) is \(4\)-colorable by \cref{T:part1}.
If \(\frac{n}{b} \geq 3\), then \(\frac{n}{c} \geq 3\), so again we are done
by \cref{T:part2}.

If \(\frac{n}{a} \geq 3\), then \(G\) is isomorphic to \(T(r, a, t)\)
for some \(0 \leq t < a\),
where \(r \geq 3\). So, we are done by \cref{T:part1}. This completes
the proof of \cref{T:main}.

\section{Summary of the colorability of \texorpdfstring{\(6\)}{6}-regular toroidal triangulations}\label{S:conclusion}

In this section, we shall present a complete picture of the colorings
of \(6\)-regular triangulations on the torus.

By the results in \cref{S:main}, we are left to classify the colorability of those \(6\)-regular
triangulations \(G\) that are either loopless multigraphs,
or isomorphic to some simple \(T(r, s, t)\) having normal circuits of lengths \(a \geq b \geq c\) such that
\((\frac{n}{a}, \frac{n}{b}) = (1, 1)\) or \((1, 2)\).

\subsection{The loopless multigraphs of the form \texorpdfstring{\(T(1, s, t)\)}{T(1, s, t)}}

Note that the graphs
\(T(1, s, t)\) and \(T(1, s, s - t - 1)\) are isomorphic by the remarks
following the statement of \cref{T:Altshuler} in \cref{S:Introduction}. So,
when \(r = 1\) we shall only focus on the values of \(t\) in the range \(0 \leq t \leq \floor{(s - 1) / 2}\).

Now, it is easy to check that \(T(1, s, t)\) has loops if and only if either \(s \leq 2\) or \(t = 0\),
and that \(T(1, s, t)\) is a loopless multigraph if and only if \(s \geq 3\) and \(t = 1, \floor{(s - 1) / 2}\).

So, we start by considering the graph \(T(1, s, 1)\) for \(s \geq 3\).
Collins and Hutchinson~\cite{ColHut99} gave explicit \(4\)-colorings of \(T(1, s, 1)\)
for \(s > 5\). Furthermore, Yeh and Zhu~\cite{YehZhu03}*{Theorem~6} observed that \(T(1, s, 1)\) is
\(3\)-chromatic if and only if \(s \equiv 0 \pmod{3}\) (after deleting
the duplicated edges in \(T(1, s, 1)\), this graph
is isomorphic to \(G_{s}[1, 2]\) in their notation).
Lastly, one can see that the graph \(T(1, 5, 1)\)
is isomorphic to \(K_{5}\) after deleting the duplicated
edges, so it is \(5\)-chromatic.

Next, we consider the graph \(T(1, s, \floor{(s - 1) / 2})\) for \(s \geq 5\).
For \(s = 2k + 1\) (\(k \geq 2\)), Yeh and Zhu~\cite{YehZhu03}*{Theorem~6} have shown that this graph is isomorphic
to \(T(1, s, 1)\) (after deleting the duplicated edges in \(T(1, s, k)\)
for \(s = 2k + 1\), this graph is isomorphic to the graph \(G_{s}[1, k]\) in their notation).
Hence, \(T(1, 2k + 1, k)\) is \(4\)-colorable for all \(k \geq 3\), and is
\(3\)-chromatic if and only if \(s \equiv 0 \equiv k - 1 \pmod{3}\), and
\(T(1, 5, 2)\) is \(5\)-chromatic since it is isomorphic to \(K_{5}\)
after deleting the duplicated edges.

For \(s = 2k + 2\) (\(k \geq 2\)), Yeh and Zhu~\cite{YehZhu03}*{Theorem~5} have shown
that \(T(1, s, \floor{(s - 1)} / 2)\) is \(4\)-colorable (and in fact \(4\)-chromatic)
if and only if \(s \equiv 0 \pmod{4}\)
(this graph is isomorphic to \(G_{s}[1, k, k + 1]\) in their notation). In this case, by removing
the duplicated edges we get a \(5\)-regular graph on the torus. So, by Brooks's theorem~\cite{Bro41},
when \(k \geq 4\) is even the graph \(T(1, 2k + 2, k)\) is \(5\)-chromatic, and
when \(k = 2\) the graph \(T(1, 6, 2)\) is isomorphic to \(K_{6}\) after deleting
the duplicated edges, and hence is \(6\)-chromatic.

\subsection{The loopless multigraphs of the form \texorpdfstring{\(T(2, s, t)\)}{T(2, s, t)}}

The graph \(T(2, s, t)\) has loops
if and only if \(s = 1\), so we assume that \(s \geq 2\). One can check that
\(T(2, s, t)\) is a loopless multigraph if and only if \(t = 0\), \(s - 2\), or \(s - 1\).
Furthermore, \(T(2, s, 0)\) and \(T(2, s, s - 2)\) are isomorphic, from the remarks following
the statement of \cref{T:Altshuler} in \cref{S:Introduction}, so there are only two cases to consider.

As observed in \cite{ColHut99}*{Theorem~3.8}, \(T(2, s, 0)\) is \(4\)-colorable
(and in fact \(4\)-chromatic) if and only if \(s \geq 2\) is even.
When \(s \geq 3\) is odd, \(T(2, s, 0)\) is isomorphic to \(T(1, 2s, \floor{(s - 1) / 2})\),
which was discussed earlier.
Next, we look at \(T(2, s, s - 1)\). This graph is isomorphic to \(T(1, 2s, 1)\) for all \(s \geq 2\),
which we have discussed earlier. So, this completes the case \(r = 2\).

\subsection{The loopless multigraphs of the form \texorpdfstring{\(T(r, s, t)\)}{T(r, s, t)} for \texorpdfstring{\(r \geq 3\)}{r >= 3}}

The graph \(T(r, s, t)\) for \(r \geq 3\)
has loops if and only if \(s = 1\), and it is a loopless multigraph
if and only if \(s = 2\). When \(t = 0\), the graph \(T(r, 2, 0)\)
is isomorphic to \(T(2, r, 0)\), which we have discussed earlier.
When \(t = 1\), the graph \(T(r, 2, 1)\) is isomorphic to \(T(1, 2r, \floor{(r - 1) / 2})\),
which was also discussed earlier.

Thus, the colorability of all the loopless multigraphs \(T(r, s, t)\) are known.

Next, we need to consider the colorability of the simple graphs \(T(r, s, t)\).
\Cref{T:main} covers the \(4\)-colorability of those \(T(r, s, t)\) that have normal
circuits of lengths \(a \geq b \geq c\) such that \((\frac{n}{a}, \frac{n}{b}) \neq (1, 1), (1, 2)\),
where \(n = rs\). So, we are only left to consider the remaining cases, namely
when \((\frac{n}{a}, \frac{n}{b}) = (1, 1)\)
or \((1, 2)\). As a step towards that, let us first consider the colorability of the
simple graphs of the form \(T(1, s, t)\).

\subsection{The simple graphs \texorpdfstring{\(T(1, s, t)\)}{T(1, s, t)}}\label{SS:T(1st)}

From the previous discussions, it suffices to consider the graphs \(T(1, s, t)\)
for those values of \(t\) in the range \(2 \leq t \leq \floor{(s - 1) / 2} - 1\). In particular,
we assume that \(s \geq 7\) in what follows.

Now, as shown in~\cite{ColHut99}*{Theorem~3.8} and discussed above in \cref{SS:counter},
the graphs \(T(1, s, 2)\) are simple
triangulations that are \(4\)-colorable (and in fact \(4\)-chromatic) if and only if \(s \equiv 0 \pmod{4}\)
since every four consecutive vertices in \(T(1, s, 2)\) induce a \(K_{4}\).
Collins and Hutchinson~\cite{ColHut99} observe that these grids are all easily
seen to be \(5\)-chromatic when \(s \geq 15\).
Explicit \(5\)-colorings for all \(s \geq 8\), \(s \neq 11\), in the spirit of Collins
and Hutchinson's work, is given as follows:
write \(s = 4u + 5v\) for \(t \geq 0\) and \(v \in \Set{0, 1, 2, 3, 4}\)
(which can be done for all \(s \geq 8\), \(s \neq 11\)),
and color \(T(1, s, 2)\) using \(u\) sets of \(1234\) followed by \(v\) sets of \(12345\). This
is easily seen to be a proper coloring of \(T(1, s, 2)\).

When \(s = 11\), the coloring
\(12345123456\) is seen to work: this is the \(6\)-chromatic graph on \(11\)
vertices found by Albertson and Hutchinson~\cite{AlbHut80}, which is also
the unique simple \(6\)-regular triangulation on the torus having \(11\) vertices,
up to isomorphism.

When \(s = 7\), \(T(1, 7, 2)\) is \(7\)-chromatic since it is isomorphic to \(K_{7}\).

Next, for each \(t \geq 3\), Collins and Hutchinson~\cite{ColHut99}*{Theorem~3.9}
exhibited \(4\)-colorings for all but finitely many of the graphs \(T(1, s, t)\) with \(t \leq \floor{(s - 1)/2} - 1\).
The remaining cases were handled by Yeh and Zhu~\cite{YehZhu03}*{Theorem~5}:
\begin{theorem}[Yeh--Zhu~\cite{YehZhu03}*{Theorem~5}, 2003]\label{T:Yeh-Zhu}
	Let \(G = T(1, s, t)\) be a simple triangulation on the torus, for \(3 \leq t \leq \floor{(s - 1)/2} - 1\).
	Then \(G\) is \(4\)-colorable, unless \(G\) satisfies one of the following conditions:
	\begin{itemize}
		\item \(s \in \Set{2t + 3, 3t + 1, 3t + 2}\) and \(s \not\equiv 0 \pmod{4}\); or
		\item \((t, s) \in \{ (3, 13)\), \((3, 17)\), \((3, 18)\), \((3, 25)\),
		\((4, 17)\), \((6, 17)\), \((6, 25)\), \((6, 33)\),
		\((7, 19)\), \((7, 25)\), \((7, 26)\), \((9, 25)\), \((10, 25)\),
		\((10, 26)\), \((10, 37)\), \((14, 33)\}\).
	\end{itemize}
\end{theorem}

Yeh and Zhu have also shown that the graphs \(T(1, s, t)\) for \(s \in \Set{ 2t + 3,
3t + 1, 3t + 2}\) are in fact isomorphic to \(T(1, s, 2)\).
Note that \(T(1, s, t)\) is isomorphic to the graph
\(G_{s}[1, t, t + 1]\) in Yeh and Zhu's notation.

\subsection{The simple graphs \texorpdfstring{\(T(r, s, t)\)}{T(r, s, t)} with normal circuits
of lengths \texorpdfstring{\(a \geq b \geq c\)}{a >= b >= c} such that
\texorpdfstring{\((\frac{n}{a}, \frac{n}{b}) = (1, 1)\)}{(n/a, n/b) = (1, 1)} or \texorpdfstring{\((1, 2)\)}{(1, 2)}}

Let \(G = (V, E)\) be a simple \(6\)-regular triangulation on the torus
with \(\card{V} = n\). Suppose that \(G\) has normal circuits of lengths
\(a \geq b \geq c\) such that \((\frac{n}{a}, \frac{n}{b}) = (1, 1)\)
or \((1, 2)\). Then, \(G\) can be represented as \(T(1, s, t)\),
and by the discussion in \cref{SS:T(1st)} we know exactly what values \(t\) can take
if \(G\) is \(5\)-chromatic. Thus, to classify the \(5\)-chromatic graphs \(G\)
satisfying \((\frac{n}{a}, \frac{n}{b}) = (1, 1)\) or \((1, 2)\),
it suffices to consider the \(5\)-chromatic graphs of the form \(T(1, s, t)\)
discussed in \cref{SS:T(1st)}
and see whether and how they can be represented as \(T(r', s', t')\) with \(r' > 1\).

This was done for the graphs \(T(1, 3s, 2)\) in \cref{S:prelim}.
A similar analysis can be done for the graphs \(T(1, 2s, 2)\) with \(s \equiv 0 \pmod{2}\)
to show that \(T(2, s, 1)\) and \(T(2, s, s - 3)\) are \(5\)-chromatic for all
odd \(s \geq 5\). These are the only cases that arise from the graphs \(T(1, s, 2)\)
for \(s \not\equiv 0 \pmod{4}\), since its normal circuits have lengths \(s\), \(s/\gcd(s, 2)\),
and \(s/\gcd(s, 3)\). This also covers the graphs \(T(1, s, t)\)
for \(s \in \Set{ 2t + 3, 3t + 1, 3t + 2}\), \(s \not\equiv 0 \pmod{4}\), \(t \geq 3\), since
they are all isomorphic to \(T(1, s, 2)\), as shown by Yeh and Zhu~\cite{YehZhu03}.

Thus, it only remains to consider the finitely many exceptional graphs
listed in the second bullet point in \cref{T:Yeh-Zhu} that have composite order.
A similar analysis can be done for these graphs as was done in \cref{S:prelim} for \(T(1, 3s, 2)\).
We omit the details and only state the final results in the next theorem.
Just one observation needs to be added before we do so: it is easy to show that a simple
graph \(T(r, s, t)\) is \(3\)-chromatic if and only if \(s \equiv 0 \equiv r - t \pmod{3}\).

We conclude this paper with a compilation of the known results
from the previous work of \cites{Hea90,Dir52c,AlbHut80,ColHut99,YehZhu03}
as well as the present work, which characterizes the colorability
of all the \(6\)-regular toroidal triangulations. We follow
the convention as adopted by \cites{ColHut99,YehZhu03} to specify
the classification by the parameters \(r, s, t\) instead of
only listing isomorphism classes of graphs.

\begin{theorem}\label{T:final}
	Let \(G = T(r, s, t)\) for \(r \geq 1\), \(s \geq 1\), \(0 \leq t \leq s - 1\)
	be a \(6\)-regular triangulation on the torus.
	If \(r = 1\), then \(T(1, s, t)\) is isomorphic to \(T(1, s, s - t - 1)\),
	so in this case consider \(t\) only in the range \(0 \leq t \leq \floor{(s - 1) / 2}\).
	\begin{enumerate}[(a)]
		\item \(G\) contains loops if and only if either \(s = 1\), or \(r = 1\) and \(s = 2\),
		or \(r = 1\) and \(t = 0\);
		
		\item \(G\) is \(7\)-chromatic if and only if \(G\) is isomorphic to \(K_{7}\),
		and this happens only when \(G = T(1, 7, 2)\);
		
		\item \(G\) is \(6\)-chromatic if and only if \(G\) is isomorphic either to \(K_{6}\)
		(after deleting duplicated edges), or to the graph of Albertson and Hutchinson~\cite{AlbHut80}
		on \(11\) vertices. The former happens only when \(G \in \{ T(1, 6, 2)\), \(T(2, 3, 0)\), \(T(2, 3, 1)\), \(T(3, 2, 0)\), \(T(3, 2, 1)\}\)
		and the latter only when \(G \in \{ T(1, 11, 2)\), \(T(1, 11, 3)\), \(T(1, 11, 4) \}\);
		
		\item \(G\) is \(5\)-chromatic if and only if \(G\) is one of the following graphs:
		\begin{itemize}
			\item \(T(1, 5, 1)\), \(T(1, 5, 2)\) (these are isomorphic to \(K_{5}\)
			after deleting duplicated edges)
			
			\item \(T(1, s, 2)\) for \(s \geq 9\), \(s \neq 11\), \(s \not\equiv 0 \pmod{4}\)
			
			\item \(T(1, s, t)\) for \(s \in \{ 2t + 2\), \(2t + 3\), \(3t + 1\), \(3t + 2 \}\),
			\(s \geq 9\), \(s \not\equiv 0 \pmod{4}\)
			
			\item \(T(2, s, 0)\), \(T(2, s, 1)\), \(T(2, s, s - 3)\), \(T(2, s, s - 2)\) for odd \(s \geq 5\)
			
			\item \(T(3, s, s - 2)\), \(T(3, s, s - 1)\) for \(s \geq 3\), \(s \not\equiv 0 \pmod{4}\)
			
			\item \(T(r, 2, 0)\), \(T(r, 2, 1)\) for odd \(r \geq 5\)
			
			\item \(T(1, s, t)\) for \((t, s) \in \{ (3, 13)\), \((3, 17)\), \((3, 18)\), \((3, 25)\),
			\((4, 17)\), \((6, 17)\), \((6, 25)\), \((6, 33)\),
			\((7, 19)\), \((7, 25)\), \((7, 26)\), \((9, 25)\), \((10, 25)\),
			\((10, 26)\), \((10, 37)\), \((14, 33)\}\)
			
			\item \(T(2, s, t)\) for \((t, s) \in \{ (3, 9)\), \((3, 13)\), \((4, 9)\), \((8, 13) \}\)
			
			\item \(T(3, s, t)\) for \((t, s) \in \{ (1, 6)\), \((2, 6)\), \((2, 11)\), \((6, 11) \}\)
			
			\item \(T(5, s, t)\) for \((t, s) \in \{ (2, 5)\), \((3, 5) \}\)
		\end{itemize}
		
		\item \(G\) is \(4\)-colorable in all other cases;
		
		\item In particular, \(G\) is \(3\)-chromatic if and only if \(s \equiv 0 \equiv r - t \pmod{3}\).
	\end{enumerate}
\end{theorem}

\end{document}